\documentclass[11pt,twoside,a4paper]{amsart}
\usepackage{amssymb}
\date{\today}

\title[Direct images of semi-meromorphic currents]{Direct images of
  semi-meromorphic currents}

\author{Mats Andersson \& Elizabeth Wulcan}
\thanks{The authors were partially supported by the Swedish Research Council.} 
\address{Department of Mathematical Sciences, Division of Mathematics, University of Gothenburg and 
Chalmers University of Technology, SE-412 96 G\"{o}teborg, Sweden}
\email{matsa@chalmers.se,  wulcan@chalmers.se}

\usepackage{geometry}
\newtheorem{thm}{Theorem}[section]
\newtheorem{lma}[thm]{Lemma}
\newtheorem{cor}[thm]{Corollary}
\newtheorem{prop}[thm]{Proposition}

\theoremstyle{definition}

\newtheorem{df}[thm]{Definition}

\theoremstyle{remark}

\newtheorem{preremark}[thm]{Remark}
\newtheorem{preex}[thm]{Example}

\newcommand{\divi}{\text{div}}

\newenvironment{remark}{\begin{preremark}}{\qed\end{preremark}}
\newenvironment{ex}{\begin{preex}}{\qed\end{preex}}

\newcommand{\C}{\mathbb{C}}

\newcommand{\dbar}{\bar{\partial}}

\newcommand{\E}{\mathcal{E}}

\newcommand{\W}{\mathcal{W}}
\newcommand{\PM}{\mathcal{PM}}

\newcommand{\Ok}{\mathcal{O}}

\newcommand{\D}{\mathcal{D}}
\newcommand{\V}{\mathcal {V}}

\newcommand{\pmm}{pseudomeromorphic }
\newcommand{\nbh}{neighborhood }
\newcommand{\1}{{\bf 1}}
\newcommand{\w}{{\wedge}}
\newcommand{\codim}{{\text{codim}\,}}

\newcommand{\Hom}{{\text{Hom}\,}}

\newcommand{\U}{{\mathcal U}}

\newcommand{\Cu}{{\mathcal C}}

\def\newop#1{\expandafter\def\csname #1\endcsname{\mathop{\rm #1}\nolimits}}
\newop{span}

\def\pfrac#1{\Big[\frac{1}{#1}\Big]}

\def\re{\text{Re\,}}

\DeclareMathOperator{\supp}{supp}
\DeclareMathOperator{\sing}{sing}

\numberwithin{equation}{section}

\begin{document}
\nocite{*}
\bibliographystyle{plain}

\begin{abstract}
We introduce a calculus for the class $ASM(X)$ of direct images of semi-meromorphic
currents on a reduded analytic space $X$, that extends the classical calculus due to Coleff, Herrera and Passare. 
Our main result is that each element in this class acts as a kind of multiplication on
the sheaf $\PM_X$ of pseudomeromorphic currents on $X$. 
We also prove that $ASM(X)$ as well as $\PM_X$ and certain subsheaves are closed under the 
action of holomorphic differential operators and interior multiplication by
holomorphic vector fields.   


\end{abstract}

\maketitle


\section{Introduction}
Let $f$ be a generically nonvanishing  holomorphic function on
 a reduced analytic space $X$ of pure dimension $n$.
It was proved by Herrera and Lieberman, \cite{HeLi}, that one can define the principal
value current 
\begin{equation}\label{apa}
\Big[\frac{1}{f}\Big].\xi:=\lim_{\epsilon\to 0}\int_{|f|^2>\epsilon}\frac{\xi}{f},
\end{equation}
for test forms $\xi$. It follows that $\dbar[1/f]$ is a current with support on the zero set
$Z(f)$ of $f$; such a current is called a residue current. Coleff and
Herrera, \cite{CoHe}, introduced products of principal value and
residue currents, like
\begin{equation}\label{apan2}
[1/f_1]\cdots [1/f_r] \dbar[1/f_{r+1}]\w\cdots 
\w\dbar[1/f_m].
\end{equation}
The product of principal value currents is commutative, but when there
are residue factors, like $\dbar [1/f_j]$, present these 
products are not (anti-)commutative in general. 
In the literature there are various generalizations and related currents, for instance
the abstract so-called Coleff-Herrera currents introduced by Bj\"ork,
see \cite{Bj}, the Bochner-Martinelli type residue currents introduced in \cite{PTY}, 
 and generalizations  in, e.g., \cite{A1},
\cite{Astrong}, and \cite{AW1}.

\smallskip
In order to obtain a coherent approach to questions about residue and
principal value currents the sheaf $\PM_X$ of {\it pseudomeromorphic
  currents} on $X$ was introduced in\cite{AW2} and further developed
in \cite{AS}; this sheaf  
consists of  direct images under holomorphic mappings of products of test forms
and currents like \eqref{apan2}. See Section~\ref{pmsection} below for the
precise definition. 
This sheaf is closed under $\dbar$ and under multiplication by smooth forms.
Pseudomeromorphic currents have a geometric
nature, similar to positive closed (or normal) currents. For example, the {\it dimension principle} states that if the  \pmm current
$\mu$ has bidegree $(*,p)$ and support on a variety of codimension larger than 
$p$, then $\mu$ must vanish. Moreover one can form restrictions $\1_W\mu$ of
the \pmm current $\mu$ to analytic (or constructible) subsets
$W\subset X$, such that 
\begin{equation}\label{skolgard}
\1_V\1_W\mu=\1_{V\cap W}\mu,
\end{equation}
see Section~\ref{restsec}. 
The notion of \pmm currents plays a
decisive role in,  for instance,
\cite{AW2, AWsemester, ASS, Lark1, Lark2, AS, Sz, Sz2, RSW, Lund, SK}.

\smallskip
It is well-known that one cannot multiply currents in general. Several attempts to find a working
calculus for principal value and residue currents have been made. A famous by Coleff and Herrera,
\cite{CoHe}, see also Passare, \cite{P}, asserts that \eqref{apan2} has all
expected (anti-)commutativity properties as long as the common zero
set of $f_1,\ldots, f_m$ has codimension $m$.
Various extension are 
introduced in the references above.  In \cite{AW2} we proved that one can give a reasonable meaning
to a product $[1/f] \mu$ for any holomorphic function $f$ and \pmm current $\mu$; 
more precisely one should consider this as an operator 
\begin{equation}\label{apan3}
\mu\mapsto [1/f] \mu
\end{equation}
on the sheaf $\PM_X$. 

\smallskip
We have not found a way 
to define a reasonable product of general \pmm currents.
Our first objective in this paper is to study a generalization of principal
value currents leading to an extension of 
\eqref{apan3}.  Following \cite{AS} we say that a current $a$ is \emph{almost
  semi-meromorphic}, $a\in ASM(X)$, if it is the direct image under a
modification of a semi-meromorphic current, i.e., a current of the form 
$\omega [1/f]$, where $f$ is a holomorphic section of a line bundle
and $\omega$ is a smooth
form with values in the same bundle.
Almost semi-meromorphic currents are \pmm and in many ways they generalize principal
value currents.  For example, it turns out that they form an (anti-)commutative
algebra, see Section ~\ref{asmsec}. Moreover $ASM(X)$ is closed under
$\partial$, see Proposition ~\ref{skrot}. Taking $\dbar$ of $a\in
ASM(X)$, however, yields an almost
semi-meromorphic current plus a residue current supported on the
\emph{Zariski singular support},  $ZSS(a)$, of $a$, which is the smallest analytic set where $a$ is not smooth. 
Many of the currents in the references above can be considered as (products of) the
residues of almost semi-meromorphic currents. 
Theorem ~\ref{hittills} states that the mapping \eqref{apan3} holds for any almost semi-meromorphic current $a$ instead of $[1/f]$. More precisely, there is a
unique extension to $X$ of the current $a\w\mu$, defined in the obvious way in
$X\setminus ZSS(a)$, such that its restriction to $ZSS(a)$ is zero.

A second objective is to prove that $\PM_X$ and $ASM(X)$ 
are closed under interior multiplication by a holomorphic
vector field $\xi$ and under the Lie derivative with respect to $\xi$;
see Sections~\ref{kolik} and ~\ref{actsec}. 

In Section ~\ref{pmsection} we recall basic known properties of the sheaf $\PM_X$ and provide 
some new results, e.g., Theorem ~\ref{grus} gives 
a new quite natural characterization of
pseudomeromorphicity. 
Section \ref{asmsec} is devoted to the study of $ASM(X)$.

\smallskip
\noindent {\bf Ackowledgment.}  We are grateful to the referee for careful reading  and pointing out unclarities  and misprints.

\section{Pseudomomeromorphic currents}\label{pmsection}
In one complex variable $s$ one can define the principal value current $[1/s^m]$ for instance as
the value 
$$
\Big[\frac{1}{s^m}\Big]=\frac{|s|^{2\lambda}}{s^m}\Big|_{\lambda=0}
$$
of the current-valued analytic continuation of $\lambda\mapsto |s|^{2\lambda}/s^m$,  a~priori defined
for $\re\lambda\gg 0$, see, e.g., \cite[Lemma~2.1]{A1}. 
We have the relations
\begin{equation}\label{utter}
\frac{\partial}{\partial s}\Big[\frac{1}{s^m}\Big]=-m\Big[\frac{1}{s^{m+1}}\Big], \quad s\Big[\frac{1}{s^{m+1}}\Big]=
\Big[\frac{1}{s^m}\Big].
\end{equation}
It is also well-known that
\begin{equation}\label{snok1}
\dbar\big[\frac{1}{s^{m}}\Big].\xi ds=\frac{2\pi i}{(m-1)!}\frac{\partial^{m-1}}{\partial s^{m-1}}\xi(0)
\end{equation}
for test functions $\xi$ and $m\geq 1$; in particular, $\dbar[1/s^{m}]$ has
support at $\{s=0\}$. Thus 
\begin{equation}\label{snok2}
\bar s \dbar\big[\frac{1}{s^{m}}\Big]=0, \quad d\bar s\w \dbar\big[\frac{1}{s^{m}}\Big]=0.
\end{equation}

We say that a function $\chi$ on the real line is a
{\it smooth approximand of the characteristic function
$\chi_{[1,\infty)}$} of the interval $[1,\infty)$, and write 
$$
\chi\sim\chi_{[1,\infty)},
$$ 
if $\chi$ is smooth,
equal to $0$ in a \nbh of $0$ and $1$ in a \nbh of $\infty$. 
It is well-known that $[1/s^m]=\lim_{\epsilon\to 0}\chi(|s|^2/\epsilon)(1/s^m)$.

\smallskip

Let $t_j$ be coordinates in an open set $\U\subset \C^N$ and let
$\alpha$ be a smooth form with compact support in $\U$.
Then 
\begin{equation}\label{1elem}
\tau=\alpha\w\Big[\frac{1}{t_1^{m_1}}\Big]\cdots\Big[\frac{1}{t_k^{m_k}}\Big]\dbar\Big[\frac{1}{t_{k+1}^{m_{k+1}}}\Big]\w\ldots
\w \dbar\Big[\frac{1}{t_r^{m_r}}\Big],
\end{equation}
where $m_1,\ldots, m_r\geq 1$, 
is a well-defined current,
since it is the tensor product of
one-variable currents (times $\alpha$). We say that $\tau$ is an {\it elementary (pseudomeromorphic) current}, and
we refer to $[1/t_j^{m_j}]$ and $\dbar[1/t_\ell^{m_\ell}]$ as its
{\it principal value factors} and  {\it residue factors}, respectively.
It is clear that \eqref{1elem} is commuting in the principal value factors and
anti-commuting in the residue factors.
We say the the intersection of $\U$ and the coordinate plane
$\{t_{k+1}=\cdots=t_r=0\}$ is the 
{\it elementary support} of $\tau$. Clearly the support of $\tau$ is  contained in
the intersection of the elementary support of $\tau$ and the support of
$\alpha$. 

\begin{remark}\label{skymma} 
Since $\partial$ does not introduce new residue factors, 
$\partial\tau$ is an elementary current, cf.\ \eqref{utter}, 
whose elementary support either equals the elementary
support $H$ of $\tau$ or is empty. Moreover $\dbar\tau$ is a finite sum of
elementary currents, whose elementary supports are either
equal to $H$ or coordinate planes of codimension $1$ in $H$, cf., ~\eqref{snok1}. 
\end{remark}

\subsection{Definition and basic properties}
Let $X$ be a reduced complex space of pure dimension $n$. 
Fix a point $x\in X$. We say that a germ $\mu$  of a current at $x$ is {\it pseudomeromorphic} at $x$,
$\mu\in \PM_x$,  if it is a finite sum
of currents of the form
\begin{equation}\label{straff}
\pi_*\tau=\pi_*^1 \cdots \pi_*^m \tau,
\end{equation}
 where $\U\subset X$ is a \nbh of $x$, 
\begin{equation}\label{1struts}
\U_m \stackrel{\pi^m}{\longrightarrow} \cdots \stackrel{\pi^2}{\longrightarrow} \U_1 
\stackrel{\pi^1}{\longrightarrow} \U_0=\U,
\end{equation}
each $\pi^j\colon \U_j \to \U_{j-1}$ is either a modification,  a simple projection
$\U_{j-1}\times Z \to \U_{j-1}$, or an open inclusion (i.e., $\U_{j}$ is an open subset
of $\U_{j-1}$), and $\tau$ is elementary on $\U_m\subset \C^N$.

By definition the union $\PM=\PM_X=\cup_x\PM_x$ is an open subset (of
the \'etal\'e space) of the sheaf $\Cu=\Cu_X$
of currents, and hence it is a subsheaf, which we call the sheaf of {\it pseudomeromorphic}  
currents\footnote{The definition here is from \cite{AS}; in the original definition in 
\cite{AW2}  simple projections were not included.}. 
A section $\mu$ of $\PM$ over an open set $\V\subset X$,  $\mu\in\PM(\V)$, is then a locally finite
sum 
\begin{equation}\label{batting}
\mu=\sum (\pi_\ell)_*\tau_\ell,
\end{equation}
where each $\pi_\ell$ is a composition of mappings as in \eqref{1struts}
(with $\U\subset\V$) and $\tau_\ell$ is elementary. For simplicity
we will always suppress the subscript $\ell$ in $\pi_\ell$. 
If $\xi$ is a smooth form, then 
\begin{equation}\label{kondor}
\xi\w \pi_*\tau=
\pi_*\big(\pi^*\xi\w\tau\big).
\end{equation}
Thus $\PM$ is closed under exterior multiplication by smooth forms.
Since 
$\dbar$ and $\partial$ commute with 
push-forwards it follows that  $\PM$ is closed under $\dbar$ and
$\partial$, cf.\ Remark ~\ref{skymma}. 

\begin{remark}\label{gryning}
Let $\tau$ be an elementary current with elementary support $H$. 
Since $H$ is the intersection of an open set $\U$ and a linear subspace, each of its components
is irreducible, and it follows that, in fact, $\tau$ is a finite sum of currents
$\tau_\ell$ such that the support of $\tau_\ell$ is contained in an
irreducible component of $H$. We may
therefore assume that each $\tau_\ell$ in \eqref{batting} has
irreducible elementary support. 


%

\end{remark}

\begin{remark}\label{rode} 
One may assume that each $\tau_\ell$ in \eqref{batting} has at
most one residue factor. Indeed, in \cite{PTY}, see also \cite[Corollary~3.5]{Atoulouse}, it is shown that
the Coleff-Herrera product 
$$
\dbar[1/t_{k+1}^{m_{k+1}}]\w\cdots\w\dbar
[1/t_{r}^{m_{r}}]
$$ 
equals the Bochner-Martinelli residue current of
$t_{k+1}^{m_{k+1}},\ldots, t_{r}^{m_{r}}$, which, see, e.g.,
\cite{A1}, is the direct image under a modification of a current of
the form $\alpha\w\dbar [1/f]$, cf., Example ~\ref{asmex} below. It follows, cf., 
\cite[Lemma 3.2]{litennot}, that \eqref{1elem} is the direct image under another modification 
of a finite sum of elementary currents with at most one residue factor. 
\end{remark}

\begin{prop}\label{1allan}
Assume that $\mu\in\PM$ has support on the subvariety $V\subset X$.

\smallskip\noindent
(i) If the holomorphic function $h$ vanishes on  $V$, then 
$\bar h \mu=0$ and $d\bar h\w\mu=0$.

\smallskip\noindent 
(ii) If $\mu$ has bidegree $(*,p)$ and $\codim V>p$, then  $\mu=0$.
\end{prop}

This proposition is from \cite{AW2}; for the adaption to nonsmooth $X$, see
\cite[Proposition~2.3]{AS}. Part (i) means
that the action of the current $\mu$ only involves holomorphic derivatives of test
forms. 
We refer to part (ii) as the \emph{dimension principle}.  
We will also need,  \cite[Proposition~1.2]{litennot}:

\begin{prop}\label{kraka}
If $\pi\colon X'\to X$ is a  modification,  
then 
$\pi_*\colon\PM(X')\to\PM(X)$
is surjective.
\end{prop}

\subsection{Basic operations on \pmm currents}\label{restsec}
Assume that  $\mu$ is \pmm  on $X$ and that $V\subset X$ is a subvariety. It was proved in 
\cite{AW2}, see also \cite{AS}, that the restriction
of $\mu$ to the open set $X\setminus V$  has a natural \pmm extension $\1_{X\setminus V}\mu$
to $X$. 
In \cite{AW2} it was obtained as the value
\begin{equation}\label{pluto}
\1_{X\setminus V}\mu:=|f|^{2\lambda}\mu|_{\lambda=0}
\end{equation}
at $\lambda=0$ of the analytic continuation of the current valued function
$\lambda\mapsto |f|^{2\lambda}\mu$, where $f$ is any tuple of
holomorphic functions such
that $Z(f)=V$. 
It follows that 
\begin{equation*}
\1_V\mu:=\mu-\1_{X\setminus V}\mu
\end{equation*}
has support on $V$. It is proved in
\cite{AW2} that this operation extends to all constructible
sets and that \eqref{skolgard} holds. 
If $\alpha$ is a smooth form, then
\begin{equation}\label{brutus1}
\1_V (\alpha\w\mu)=\alpha\w \1_V\mu.
\end{equation}
Moreover, 
if $\pi\colon X'\to X$ is a modification, a simple projection or an open
inclusion and $\mu=\pi_*\mu'$, then
\begin{equation}\label{brutus2}
\1_V\mu=\pi_*\big(\1_{\pi^{-1}V}\mu'\big).
\end{equation}

In this paper it is convenient to express $\1_{X\setminus V}\mu$ as a limit of currents that
are \pmm themselves. 

\begin{lma}\label{3apsko}
Let $V$ be a germ of a subvariety at $x\in X$, let $f$ be a tuple of 
holomorphic  functions whose
common zero set is precisely $V$, let $v$ be a positive and smooth function, and let $\chi\sim\chi_{[1,\infty)}$.
For each germ of a \pmm current $\mu$ at $x$ we have
\begin{equation}\label{restrikdef2}
\mathbf{1}_{X\setminus V}\mu= \lim_{\epsilon\to 0} \chi(|f|^2v/\epsilon) \mu.
\end{equation}
\end{lma}

Because of the factor $v$, the lemma holds just as well for a holomorphic section
$f$ of a Hermitian vector bundle.

In case $V$ is a hypersurface and $f$ is one single holomorphic function, or
section of a line bundle, the lemma
follows directly from Lemma~6 in \cite{LS} by just taking $T=f\mu$.  We will reduce the general case to this lemma.
The proof of this lemma relies on the proof of Theorem~1.1 in
\cite{LS}, which is quite involved. 
For a more direct proof of Lemma~\ref{3apsko}, see
the proof of Proposition~3.4 in \cite[Ch.2]{Abook}.

\begin{proof} Let $\pi\colon X'\to X$ be a smooth modification such that $\pi^* f=f^0f'$, where
$f^0$ is a holomorphic section of a Hermitian line bundle $L\to X'$
and $f'$ is a nonvanishing tuple of holomorphic sections
of $L^{-1}$. In view of Proposition~\ref{kraka} we  can assume that
$\mu=\pi_*\mu'$, where $\mu'$ is pseudomeromorphic on $X'$.
Then 
$$
|\pi^* f|^2\pi^*v=|f^0|^2 |f'|^2\pi^*v,
$$
and from \cite[Lemma~6]{LS} we thus have that
$$
\lim_{\epsilon\to 0}\chi(|\pi^* f|^2\pi^*v/\epsilon)\mu'=\1_{X'\setminus\pi^{-1}V}\mu'.
$$
In view of \eqref{brutus2} we get \eqref{restrikdef2}.
\end{proof}

\begin{remark} Lemma~\ref{3apsko} holds even if
  $\chi=\chi_{[1,\infty)}$. However, in  general it is not  
obvious what $\chi(|f|^2v/\epsilon)\mu$ means. Let $\chi^\delta$ be smooth approximands
such that  $\chi^\delta\to \chi_{[1,\infty)}$.  
It follows from the proof of Lemma~6 in \cite{LS} that for small
enough $\epsilon$, depending on $\mu$, $f$, and $v$, the limit 
$\lim_{\delta\to 0}\chi^\delta(|f|^2v/\epsilon)\mu$ exists and is independent of the choice of $\chi^\delta$; thus we can take it
as the definition of $\chi(|f|^2v/\epsilon)\mu$.  
In fact, it turns out that after a suitable change of real coordinates one can
realize $\chi(|f|^2v/\epsilon)\mu$ as a tensor product of two
currents. 
%
In particular
we get 
\[
\chi(|f|^2/\epsilon)\frac{1}{f}.\xi=\int_{|f|^2>\epsilon}\frac{\xi}{f},
\]
cf., \eqref{apa}. 
\end{remark}

We will need the following observation.

\begin{lma}\label{batong}
If $\mu$ has the form \eqref{batting},  then
$$
\1_V\mu=\sum_{\supp\tau_\ell\subset\pi^{-1}V} \pi_*\tau_\ell.
$$
\end{lma}

It follows from the proof below that we just as well can take the sum over all $\ell$ such that
the elementary supports of $\tau_\ell$ are  contained in $\pi^{-1}V$. 

\begin{proof}
In view of \eqref{brutus2} we have that
$$
\1_V\mu=\sum_\ell \pi_*\big(\1_{\pi^{-1}V} \tau_\ell\big).
$$
If $\supp\tau_\ell\subset \pi^{-1}V$,  then clearly 
$\1_{\pi^{-1}V} \tau_\ell=\tau_\ell$.  
We now claim that if $\supp\tau_\ell$ is not contained
in $\pi^{-1}V$, then 
$\1_{\pi^{-1}V} \tau_\ell=0$. If $\supp \tau_\ell \not\subset
\pi^{-1}V$,  the elementary support $H$ of $\tau_\ell$
is not contained in  $\pi^{-1}V$. Assume that $H$ has codimension
$q$. Then $\tau_\ell$ is of the form $\tau_\ell=\alpha\wedge\tau'$,
where $\alpha$ is smooth and $\tau'$ is elementary of bidegree
$(0,q)$. It follows from \eqref{brutus1} that
$$
\1_{\pi^{-1}V}\tau_\ell=\alpha\w \1_{\pi^{-1}V}\tau'.
$$
By Remark ~\ref{gryning} we may assume that $H$ is irreducible, 
and therefore
$\pi^{-1}V\cap H$  has codimension at least $q+1$ in $\U$.
Since $\1_{\pi^{-1}V}\tau'$ has support on $\pi^{-1}V\cap H$ it must
vanish in view of 
the dimension principle.
Thus the lemma follows.
\end{proof}

We now consider another fundamental operation on $\PM$ introduced in \cite{AW2}.

\begin{prop}[\cite{AW2}]\label{groda1}
Given  a holomorphic function $h$ and a \pmm current $\mu$ there is a \pmm current
$T$ such that $T=(1/h)\mu$ in the open set where $h\neq 0$ and $\1_{\{h=0\}}T=0$. 
\end{prop}
Here $h$ may just as well be a holomorphic section of a line bundle. 
Clearly this current $T$ must be unique and we denote it by 
$[1/h] \mu$. In \cite{AW2} the current $[1/h]\mu$ was defined as
$(|h|^{2\lambda}\mu/h)|_{\lambda=0}.$
\begin{remark}\label{linne}
Notice that\footnote{We have not exluded the
possibility that $h$  vanishes identically on some (or all) irreducible components of $X$.}
$
h[1/h]\mu=\1_{\{h\neq 0\}}\mu; 
$ in particular, $h[1/h]\mu\neq \mu$ in general. For example, 
$z[1/z]\dbar[1/z]=0$.
\end{remark}

Since $[1/h]\mu=(1/h)\mu$ in $\{h\neq 0\}$ and
$[1/h]\mu=\1_{\{h\neq 0\}}[1/h]\mu$,   
it follows from \eqref{restrikdef2}  that
\begin{equation}\label{snart}
\Big[\frac{1}{h}\Big]\mu=\lim_{\epsilon\to 0}\chi(|h|^2v/\epsilon)\frac{1}{h}\mu. 
\end{equation}
One can also define
\begin{equation}\label{lei}
\dbar\Big[\frac{1}{h}\Big]\w\mu:=\dbar\Big(\Big[\frac{1}{h}\Big]\mu\Big)-\Big[\frac{1}{h}\Big]\dbar\mu,
\end{equation}
i.e., so that "Leibniz's rule" holds.
Notice that if $\pi\colon X'\to X$ is a modification and $\mu=\pi_*\mu'$, then
\begin{equation}\label{polly}
\Big[\frac{1}{h}\Big]\mu=\pi_*\Big(\Big[\frac{1}{\pi^*h}\Big]\mu'\Big), \quad 
\dbar\Big[\frac{1}{h}\Big]\w\mu=\pi_*\Big(\dbar\Big[\frac{1}{\pi^*h}\Big]\w \mu'\big).
\end{equation}
This follows, e.g.,  from \eqref{kondor} and \eqref{snart}. 
It is also readily checked  that  
\begin{equation}\label{lei1}
\dbar\Big(\dbar\Big[\frac{1}{h}\Big]\w\mu\Big)=-\dbar\Big[\frac{1}{h}\Big]\w\dbar\mu.
\end{equation}
\begin{remark}\label{alligator} 
Since $[1/f][1/g]=[1/(fg)]=[1/g][1/f]$ it follows from \eqref{lei} that
$$
\dbar\Big[\frac{1}{f}\Big]\cdot\Big[\frac{1}{g}\Big]+\Big[\frac{1}{f}\Big]\dbar\Big[\frac{1}{g}\Big]=\dbar\Big[\frac{1}{g}\Big]\cdot\Big[\frac{1}{f}\Big]+
\Big[\frac{1}{g}\Big]\dbar\Big[\frac{1}{f}\Big].
$$
However, it is not true in general that $[1/g]\dbar[1/f]=\dbar[1/f]\cdot[1/g]$. For instance,
$[1/z]\dbar[1/z]=0$, whereas $\dbar[1/z]\cdot [1/z]=\dbar[1/z^2]$. 
\end{remark}

\smallskip
We now consider 
tensor products and direct images under simple projections.

\begin{lma}\label{1tensor}
If $\mu\in\PM_X$ and $\mu'\in\PM_{X'}$, then $\mu\otimes \mu'\in\PM_{X\times X'}$.
\end{lma}

This is precisely \cite[Lemma 3.3]{litennot}.
It is easy to verify that
\begin{equation}\label{pelargonia2}
\1_{V\times V'} \mu\otimes \mu'=\1_V \mu\otimes \1_{V'} \mu'.
\end{equation}

\begin{lma}\label{bock}
Assume that $p\colon Z\times W\to Z$ is a simple projection. If $\mu$ is in $\PM_{Z\times W}$ and
$p^{-1} K\cap \supp \mu$ is compact for each compact set $K\subset Z$, then
$p_*\mu$ is in $\PM_Z$.  
\end{lma}

\begin{proof}
Since pseudomeromorphicity is a local property, after 
multiplying $\mu$ if necessary by a suitable cutoff function we can 
assume that $\mu$ has compact support. By compactness and a partition of
unity we then have  a finite
representation $\mu=\sum_\ell \pi_*\tau_\ell$.  
Now the lemma follows from the very definition of $\PM$.
\end{proof}

\begin{ex}
Assume that $\tau$ is an elementary current on $X$, $p$ is a simple
projection $X\times X'\to X$, and $\chi$ is any test form in $X'$
with total integral $1$.
Then the tensor product $\tau\otimes\chi$ is an elementary current
in $X\times X'$ such that $p_*(\tau\otimes\chi)=\tau$.
\end{ex}

The following result provides a new, quite natural definition of pseudomeromorphicity.

\begin{thm}\label{grus}
(i) \  Assume that $X$ is smooth. Then a  germ of a current $\mu$ at $x\in X$ is \pmm if and only if
it is a finite sum
\begin{equation}\label{apa1}
\mu =\sum_\ell (f_\ell)_* \tau_\ell,
\end{equation}
where $f_\ell\colon \U_\ell\to X$ are holomorphic mappings and $\tau_\ell$ are elementary.
\smallskip

\noindent(ii)  If $X$ is a reduced space of pure dimension and $\pi\colon X'\to X$ is a smooth
modification, then a current $\mu$ on $X$ is pseudomeromorphic if and only if
there is a \pmm current $\mu'$ on $X'$ such that $\mu=\pi_*  \mu'$.
\end{thm}

\begin{proof}
By definition a germ of a \pmm current is of the form \eqref{apa1}.   Now assume that $f\colon \U\to X$ is any holomorphic mapping and $\tau$ is elementary in $\U\subset \C^N$.   Let $F\colon \U\to \U\times X$ be the
mapping $F(s)=(s,f(s))$.   Let $\widetilde F$ be $F$ considered as a biholomorphism onto the graph $\Gamma\subset \U\times X$ and let $i\colon \Gamma\to \U\times X$ be the natural injection. Then clearly $\widetilde F_*\tau$ is
\pmm on $\Gamma$ and in view of  \cite[Theorem 1.1(i)]{litennot},   $F_*\tau=i_*\widetilde F_*\tau$ is \pmm in
$\U\times X$. Clearly, it has compact support in $\U\times X$. If $p$ is the projection $\U\times X\to X$, we can therefore apply Lemma ~\ref{bock}, and conclude that $f_*\tau=p_*F_*\tau$ is \pmm in $X$.  Thus
part (i) is proved.  
Part (ii)  is just  Proposition ~\ref{kraka}. 
\end{proof}

\begin{cor}\label{grus1} 
Assume that  $f\colon W\to X$ is a holomorphic mapping and $X$ is smooth. If 
$\mu$ is \pmm on $W$ with compact support, then $f_*\mu$ is \pmm on $X$.
\end{cor}

\begin{proof} 
We may assume that $\mu=\pi_*\tau$, where $\pi\colon\U\to W$ is a mapping as in the definition of pseudomeromorphicity and $\tau$ is elementary in $\U$. Then we can apply 
Theorem ~\ref{grus} (i) to the mapping $f\circ\pi\colon \U\to X$.
It follows 
that $f_*\mu=f_*\pi_*\tau=(f\circ\pi)_*\tau$ is \pmm in $X$.
\end{proof}

\begin{remark} Notice that in the proof of Theorem ~\ref{grus} 
we only used \cite[Theorem 1.1(i)]{litennot}, which asserts that $i_*$ maps
$\PM_W$ into $\PM_X$ if $i:W\to X$ is an embedding of a reduced
pure-dimensional space $W$ into a manifold $X$, in the relatively simple case
when $W$ is a smooth submanifold. The general case now follows from Corollary ~\ref{grus1}. 
Part (ii) of \cite[Theorem 1.1]{litennot} is a partial converse:
{\it If $\mu=i_*\nu$ is \pmm in $X$ and $\1_{W_{sing}}\mu=0$,
then $\nu$ is \pmm on $W$.}  
The proof of this fact relies on the possibility to make a so-called strong resolution. This means that   there is
a resolution $X'\to X$ that is a biholomorphism outside $W$, and such that the strict transform of $W$ is
a smooth resolution of $W$. 
\end{remark}

\section{Action of holomorphic differential operators and vector fields}\label{kolik}
Let $X$ be a reduced analytic space of pure dimension. 
We already know that $\partial$ maps $\PM_X$ into itself. We shall now consider a more
general statement, and to this end we need the following result that is interesting in itself.

\begin{prop}\label{1ekorre} 
Assume that $\mu\in\PM_x$ where $x\in X$. 
If $h\in\Ok_x$ is not identically zero on any irreducible component of
$X$ at $x$, 
then there is $\mu'\in\PM_x$ such that
$h\mu'=\mu$.   
\end{prop}

\begin{remark}\label{los} 
By a partition of unity we can get a global such $\mu'$ if $\mu$ and
$h$ are global. If $\mu$ has compact support in $\U\subset X$ we can
choose $\mu'$ with compact support in $\U$. 
\end{remark}

\begin{remark} If $\mu$ has support on $V$ we may assume as well that $\mu'$ has. Indeed,
$\mu=\1_V\mu=\1_V h\mu'=h\1_V\mu'$, so we can replace a given solution $\mu'$ by $\1_V\mu'$.
\end{remark}

\begin{ex} Proposition ~\ref{1ekorre}  is not true if $h$ is anti-holomorphic. In fact, if
$\bar z \mu'=1$, then  $[1/z]\mu'$ is equal to $1/|z|^2$ outside $0$.
Thus $\lim_{\epsilon\to 0}\chi(|z|^2/\epsilon)\mu'/z$ does not exist, and hence $\mu'$ cannot
be pseudomeromorphic, cf.,  Proposition ~\ref{groda1} and \eqref{snart}. 
\end{ex}

%
\begin{proof}[Proof of Proposition~\ref{1ekorre}] 
First assume that $\tau$ is an elementary \pmm current in $\C^N_t$ and
$h$ is a monomial.  By induction it is enough to assume that $h=t_1$.
If $t_1$ is a residue factor 
in $\tau$,
then we just raise the power of $t_1$ in that factor one unit. Otherwise
we take $\tau'=(1/t_1)\tau$. Then $h\tau'=\tau$. 

We may assume that 
$\mu=\pi_* \tau$, where $\pi:\U\to X$ and $\tau$ is elementary of the form \eqref{1elem}. 
By Hironaka's theorem we can find a modification $\nu:\U'\to \U$ such
that, locally in $\U'$, $\nu^*\pi^* h$ is a monomial and $\nu^* t_j$ are monomials (times
nonvanishing functions). By a partition of unity in $\U'$ and repeated use of \eqref{polly} 
it follows that $\tau$ is a finite sum of currents $\nu_*\tau'$, where 
\begin{equation*}
\tau':=\nu^*\alpha\w\Big[\frac{1}{\nu^*t_1^{m_1}}\Big]\cdots\Big[\frac{1}{\nu^*t_k^{m_k}}\Big]\dbar\Big[\frac{1}{\nu^*t_{k+1}^{m_{k+1}}}\Big]\w\ldots
\w \dbar\Big[\frac{1}{\nu^*t_r^{m_r}}\Big].
\end{equation*}
Each such term is a sum of elementary currents $\tau_\ell$
in view of \eqref{lei}. 
By the first part of the proof there are elementary currents $\tau_\ell'$ in $\U'$ such that
$\nu^*\pi^* h ~\tau_\ell'=\tau_\ell$.  Now the proposition follows in view of
\eqref{kondor}.
\end{proof}

\begin{thm}\label{storknar2}
Assume that $X$ is smooth 
at $x\in X$.  

\noindent (i)
If  $z$ is a local holomorphic coordinate system at $x$ and 
\begin{equation}\label{pluck}
\mu=\sum'_{|I|=p}\mu_I\w dz_I
\end{equation}
is a germ in $\PM_x$, then each $\mu_I$ is in $\PM_x$.  

\smallskip
\noindent (ii)\ If $\xi$ is a germ of a holomorphic vector field, then 
the contraction $\xi\neg \mu$ and the Lie derivative $L_\xi \mu$ are in $\PM_x$.  
\end{thm}

Notice that (ii) is not true for anti-holomorphic vector fields.
For example,  
$\mu=
(\partial/\partial \bar z)\neg\dbar(1/z)$
is a nonzero current of degree $0$ with support at $0$. In view of the dimension principle,
it cannot be pseudomeromorphic.

\begin{proof}
We will first assume that $\mu$ has bidegree $(n,*)$ so that
$\mu=\hat\mu\w dz$, where $\hat\mu$ has bidegree $(0,*)$, and show that $\hat\mu$ is pseudomeromorphic. We may assume that
$\mu=\pi_*(\tau\w ds)$, where $\pi:\U\to X$ is a mapping as in the
definition of pseudomeromorphicity, $s$ are local coordinates
in $\U\subset \C^m$, and $\tau$ is
elementary. Since $\pi$ has generically surjective differential, 
we can write $s=(s',s'')=(s'_1,\ldots, s'_n, s''_{n+1},\ldots, s''_m)$
so that $h:=\det(\partial \pi/\partial s')=\det(\partial z/\partial
s')$ is generically nonvanishing in $\U$. 
By Proposition~\ref{1ekorre} and Remark~\ref{los} there is a
pseudomeromorphic 
$\tau'$ with compact support in $\U$ such that $h\tau'=\tau$ in $\U$. 
Now  
$$
\hat \mu \w dz=\pi_*(\tau\w ds)=\pi_*(\tau'\w hds'\w ds'')=
\pi_*(\tau'\w\pi^* dz\w ds'')=\pm \pi_*(\tau'\w ds'')\w dz.
$$
Thus $\hat\mu=\pm \pi_*(\tau'\w ds'')$ is pseudomeromorphic.  
In general, $\mu_I\w dz=\pm \mu\w dz_{I^c}$, where $I^c$ is the complementary multiindex
of $I$.  It follows from above that $\mu_I$ is pseudomeromorphic. 
Thus (i) follows.

The first statement of (ii) follows immediately from (i), and the second one follows
since $L_\xi \mu=\partial(\xi\neg\mu)+\xi\neg(\partial\mu)$.
\end{proof}

\subsection{The sheaves $\PM_X^Z$ and $\W_X^Z$}
Let $X$ be a reduced analytic space, let $Z\subset X$ be a (reduced) subspace
of pure dimension, and denote by 
$\PM_X^Z$ the subsheaf of $\PM_X$ of currents that have support on $Z$.
We say that $\mu\in\PM_X^Z$  has the {\it standard extension property, SEP, on  
$Z$} if $\1_W\mu=0$ in $\U$ for each  
subvariety $W\subset \U\cap Z$ of positive codimension, where $\U$ is any open set in $X$.
Let $\W_X^Z$ be the subsheaf of $\PM_X^Z$ of currents with the SEP on
$Z$.  
In case $Z=X$ we usually write $\W_X$ rather than $\W_X^X$.

\begin{ex}\label{elex}
Note that an elementary current in $\U$ with elementary
support $H$ is in $\W_\U^H$. 
\end{ex}

It is easy to see that Theorem ~\ref{storknar2} holds for $\PM_X^Z$ as well, since neither $\partial$
nor contraction can increase support.  Somewhat less obvious is that also the SEP is preserved.

\begin{thm}\label{storknar3}  
The sheaf $\W_X^Z$ is invariant under $\partial$, and 
the statements in Theorem ~\ref{storknar2} hold 
for $\W_X^Z$ instead of $\PM$. 
\end{thm}
 
This theorem is a consequence of the following general equalities.

\begin{prop}
Assume that $\mu$ is a \pmm current on $X$. 
If $V\subset X$ is any analytic subset, then
\begin{equation}\label{puff2}
\1_V\partial \mu=\partial \1_V\mu.
\end{equation}
If $\xi$ is a holomorphic vector field, then
\begin{equation}\label{puff1}
\1_V\xi\neg\mu=\xi\neg\1_V\mu.
\end{equation}
\end{prop}

\begin{proof}
Note that \eqref{puff1} follows in view of \eqref{restrikdef2}. Let us
therefore 
 focus on  \eqref{puff2}.
By \eqref{skolgard} it is enough to consider $V=Z(h)$, where $h$ is a nontrivial
holomorphic function.  
Take $\chi\sim \chi_{[1,\infty)}$ and let $\chi_\epsilon=\chi(|h|^2/\epsilon)$.   Now
\begin{equation}\label{kross}
\chi_\epsilon\partial \mu=\partial(\chi_\epsilon\mu)-\partial\chi_\epsilon\w \mu.
\end{equation}
If the last term  tends to $0$ when $\epsilon\to 0$, after taking limits we get that $\1_{h\neq 0}\partial\mu=
\partial(\1_{h\neq 0}\mu)$, which is equivalent to \eqref{puff2}.   
Let $\hat\chi(t)= t\chi'(t)+ \chi(t) $, and notice that also $\hat\chi\sim  \chi_{[1,\infty)}$.
According to Proposition ~\ref{1ekorre} there is a \pmm $\mu'$ such that $\mu=h\mu'$. 
The last term in \eqref{kross} is therefore
$$
\chi'(|h|^2/\epsilon)\bar h\partial h\w\mu/\epsilon=
\chi'(|h|^2/\epsilon)|h|^2\partial h\w\mu'/\epsilon=
\hat\chi(|h|^2/\epsilon) \partial h\w \mu'- \chi_\epsilon \partial h\w\mu',
$$
which tends to $\1_{h\neq 0}\partial h\w \mu'-\1_{h\neq 0}\partial h\w \mu'=0$.
\end{proof}


\section{Almost semi-meromorphic currents}\label{asmsec}
We say that a current on $X$ is \emph{semi-meromorphic} if it is of the form 
$\omega [1/f]$, where $f$ is a generically nonvanishing holomorphic section of a line
bundle $L\to X$ and  $\omega$ is a smooth form with values in $L$. 
For simplicity we will often omit the brackets $[\  ]$ indicating principal value in the sequel.
Since furthermore $\omega [1/f]=[1/f]\omega$ when $\omega$ is smooth we can write just
$\omega/f$.
   
\subsection{The algebra $ASM(X)$}
Let $X$ be a pure-dimensional reduced analytic space.  We say that a current $a$ is {\it almost
semi-meromorphic} in $X$, $a\in ASM(X)$,  if there is a modification $\pi\colon X'\to X$ such that
\begin{equation}\label{asm}
a=\pi_*(\omega/f),
\end{equation}
where $\omega/f$ is semi-meromorphic in $X'$.
We say that
$a$ is {\it almost smooth} in $X$ if one can choose $f$ to be
nonvanishing.  
We can assume that $X'$ is smooth because otherwise we take a smooth modification
$\pi'\colon X''\to X'$ and consider the pullbacks of $f$ and $\omega$ to $X''$,
cf., \eqref{polly}.   If nothing else is said we tacitly assume that $X'$ is smooth.

Notice that if $\U\subset X$ is an open subset, then the restriction $a_\U$ of $a\in ASM(X)$ to
$\U$ is in $ASM(\U)$. In fact,  if \eqref{asm} holds, then 
$\U':=\pi^{-1}\U\to \U$ is a modification of $\U$,  and $a_\U$ is the direct image of 
the restriction of $\omega/f$ to $\U'$. 
 
If $V$ has positive codimension in $\U\subset X$,
then $\pi^{-1} V$ has positive codimension in $\U'$ and
$\1_V a=\pi_*(\1_{\pi^{-1}V} (\omega/f))=\pi_*(\omega\1_{\pi^{-1}V}
(1/f))=0$ in $\U$, cf., 
\eqref{brutus2}, \eqref{brutus1}, and the dimension principle.  Thus $ASM(X)$ is
contained in $\W(X)$.


\begin{remark} One can introduce a notion "locally almost semi-meromorphic
current" and consider the associated sheaf. However, for the moment we have no need for such a 
concept. 
\end{remark}

\begin{ex}
Assume that $X=\{zw=0\}\subset\C^2$. Let $a\colon X\to\C$  
be $1$ and $0$ on the $z$-axis and the $w$-axis, respectively, except at the origin.
Then $a$ is almost smooth. Indeed the normalization
$\nu\colon\widetilde X\to X$ consists of two disjoint components and
$a=\nu_*\tilde a$, where $\tilde a$ is $0$ and $1$, respectively, on these components.
\end{ex}

Given a modification $\pi \colon X'\to X$, let $\sing(\pi)\subset X'$ be the (analytic) set 
where $\pi$ is not a
biholomorphism.  By the definition of a 
modification it has positive codimension. Let $a$ be given by
\eqref{asm} and let $Z\subset X'$   be the zero set of $f$. By assumption 
also $Z$ has positive codimension.  Notice that $a\in ASM (X)$ is smooth outside
$\pi(Z\cup \sing(\pi))$ which has positive codimension in $X$. 
We let $ZSS(a)$, the {\it Zariski-singular support} of $a$, be the smallest 
Zariski-closed set $V\subset X$ such that $a$ is smooth outside $V$.

\begin{ex}
Assume that $a\in ASM(X)$ is almost smooth. Then $a=\pi_*\omega$,
where $\omega$ is smooth, and thus $ZSS(a)\subset \pi(\sing(\pi))$.
This inclusion may be strict.  For example if $a$ is smooth, then 
$ZSS(a)$ is empty. In this case $\omega=\pi^* a$ outside $\sing(\pi)$ and since both sides are smooth across $\sing(\pi)$, by continuity, then $\omega=\pi^* a$ everywhere in $X'$.
\end{ex}

Given two modifications $X_1\to X$ and $X_2\to X$, there is a modification $\pi\colon X'\to X$ that
factorizes over both $X_1$ and $X_2$, i.e., we have $X'\to X_j\to X$ for $j=1,2$. Therefore,
given $a_1,a_2\in ASM(X)$ we can assume that $a_j=\pi_*(\omega_j/f_j)$, $j=1,2$.
It follows that 
$$
a_1+a_2=\pi_*\Big(\frac{\omega_1}{f_1}+\frac{\omega_2}{f_2}\Big)=
\pi_*\frac{f_2\omega_1+f_1\omega_2}{f_1f_2},
$$
so that $a_1+a_2$ is in $ASM(X)$ as well. Moreover, 
$A:=\pi_*(\omega_1\w\omega_2/f_1 f_2)$
is an almost semi-meromorphic current that coincides with $a_1\w a_2$ outside the set
$\pi\big (\sing(\pi)\cup  V(f_1)\cup V(f_2)\big)$.  If we had 
 other representations $a_j=\pi'_*(\omega'_j/f_j')$, $j=1,2$, we would get an almost semi-meromorphic $A'$ that coincides generically with
$a_1\w a_2$ on $X$. Since almost semi-meromorphic have the SEP, thus $A=A'$. Hence we can define 
$a_1\w a_2$ as  $A$. 
Similarly, since 
$$
a_2\w a_1= (-1)^{\deg a_1 \deg a_2} a_1\w a_2,\quad a_1\w(a_2+a_3)=a_1\w a_2+a_1\w a_3
$$
and  
$$
a_1\w (a_2\w a_3)=(a_1\w a_2)\w a_3
$$
hold generically on $X$ and because of the SEP they hold on $X$. 
Thus $ASM(X)$ is an algebra.

 \begin{remark} Notice that the almost smooth currents form a subalgebra of $ASM(X)$.
\end{remark}

\begin{ex} 
Clearly $ZSS(a_1\w a_2)\subset ZSS(a_1)\cup ZSS(a_2)$ but the inclusion
may be strict.  Take for instance $z_1/z_2$ and $z_2/z_3$.   
\end{ex}

\begin{ex}\label{meromorphic} 
The most basic example of an (almost semi-)meromorphic current is the
principal value current associated with a meromorphic form. 
Let $f$ a be meromorphic $k$-form on $X$, i.e., locally $f=g/h$ where $h$ is a holomorphic 
function that is generically nonvanishing and 
$g$ is a holomorphic $(k,0)$-form. 
By definition 
$g/h=g'/h'$ if and only if $g' h-gh'$ vanishes outside a set of positive codimension. 
In that case 
\begin{equation}\label{3bob}
g\pfrac{h}=g'\pfrac{h'}
\end{equation}
outside a set of positive codimension. By the dimension principle
therefore \eqref{3bob} holds everywhere. Thus there is a well-defined 
almost semi-meromorphic current  $[f]$ associated with $f$. 
Notice that $ZSS([f])$ is contained
in the pole set of the meromorphic form $f$, so unless $X$ is smooth it may have
codimension larger than $1$.
Actually,  $ZSS([f])$ is equal to the pole set of $f$.  In fact, by continuity $\dbar f=0$ where
$f$ is smooth, and by a classical result proved by Malgrange (at
least for functions), \cite{M}, then
$f$ is holomorphic there.
 \end{ex}


The following lemma will be crucial in what follows.

\begin{lma}\label{asm0}
If $a$ is almost semi-meromorphic in $X$, then there is a representation 
\eqref{asm} such that $f$ is nonvanishing in $X'\setminus
\pi^{-1}ZSS(a)$. 
\end{lma}

\begin{proof}
Let $V=ZSS(a)$ and 
assume that we have a representation \eqref{asm} and that $X'$ is smooth. Let $Z$ be the 
union of the irreducible components of the divisor defined by $f$ that are not fully contained in  
$\pi^{-1}V$.  
 Since $X'$ is smooth,  $Z$ is a Cartier divisor and thus the divisor of a section $f'$ of some line bundle $L'\to X'$.  It follows that $g:=f/f'$ is a holomorphic section of $L\otimes (L')^{-1}$
in $X'$ that is nonvanishing in $X'\setminus \pi^{-1}V$.
Outside  $\sing(\pi)\cup Z\cup \pi^{-1} V$ we have that 
\begin{equation}\label{asm2}
\omega=f\pi^*a=f'g \pi^*a.
\end{equation}
By continuity,  \eqref{asm2} must hold in $X'\setminus \pi^{-1} V$
since  both sides are smooth there.  

We claim that $\widetilde \omega := \omega/f'$ is smooth in $X'$.  
Taking this for granted, then 
\begin{equation}\label{asm5}
\pi_*\frac{\widetilde\omega}{g}
\end{equation}
is in $ASM(X)$ and the zero set of $g$ is contained in $\pi^{-1} V$. Since \eqref{asm5} coincides
with $a$  outside $V\cup \pi(\sing(\pi))$  
it follows by the SEP that \eqref{asm5} indeed is equal to $a$ in $X$. Thus the lemma follows.

The claim is a local statement in $X'$ so given a point in $X'$ we can choose local coordinates $t$ 
in a \nbh $\U$ of that point and consider
each coefficient of the form $\omega$ with respect to these coordinates. Thus we may assume
that $\omega$ is a function and that  $\omega=f'\gamma$ where $\gamma=g\pi^*a$ is smooth in
$\U\setminus \pi^{-1} V$, cf., \eqref{asm2}  and the comment thereafter.
For all multiindices $\alpha$  thus 
\begin{equation}\label{asm3}
\frac{\partial^{\alpha}\omega}{\partial\bar t^{\alpha}}\dbar\frac{1}{f'}=0
\end{equation}
in $\U\setminus\pi^{-1} V$, since $f'\dbar (1/f')=0$.  By assumption $Z\cap \pi^{-1} V$ has positive codimension in
$Z$. By the dimension principle it follows that \eqref{asm3} holds in
$\U$ for all $\alpha$, since $\dbar(1/f')$ has support on $Z$. 
From \cite[Theorem~1.2]{Aglatt}  we conclude that $\widetilde \omega$ is smooth in $\U$.  It follows that
$\widetilde \omega$ is smooth in $X'$.   
\end{proof}

\subsection{Action of $ASM(X)$ on $\PM_X$}

We will now extend Proposition ~\ref{groda1}  to general almost semi-meromorphic
currents.

\begin{thm}\label{hittills}
Assume that  $a\in ASM(X)$. For each $\mu\in\PM(X)$ 
there is a unique \pmm  current $T$ in $X$ that coincides with 
$a\w\mu$ in $X\setminus ZSS(a)$ and such that 
$\1_{ZSS(a)} T=0$.
\end{thm}

Let $V=ZSS(a)$. 
If such an extension $T$ exists then 
$T=\1_{X\setminus V}T=\1_{X\setminus V}a\w\mu$ and so $T$ is
unique. Moreover, if $h$ is a holomorphic tuple such that  $Z(h)=V$, then 
\begin{equation}\label{asm4}
T=\lim_{\epsilon\to 0}\chi(|h|^2v/\epsilon) a\w\mu 
\end{equation}
 in view of Lemma ~\ref{3apsko}.
We will denote the extension $T$ by  $a\w\mu$ as well.

\begin{proof} 
As observed above, if the extension $T$ exists, then \eqref{asm4}
holds. 
Conversely, 
if the limit in \eqref{asm4} exists as a \pmm  current $T$ on $X$, then 
it must coincide with $a\w\mu$ in $X\setminus V$. In particular, 
$\chi(|h|^2v/\epsilon) T=\chi(|h|^2v/\epsilon) a\w\mu$ for each
$\epsilon>0$ and hence, taking limits and using Lemma ~\ref{3apsko},
we get 
$\1_{X\setminus V} T=T$, i.e., $\1_{ZSS(a)} T=0$.
To prove the theorem it is thus enough to verify that the limit in \eqref{asm4} exists
as a \pmm current.

In view of  Lemma ~\ref{asm0} we may assume that $a$ has the form
\eqref{asm}, where $Z=Z(f)$ is contained in $\pi^{-1} V$ and  
$\omega/f=\pi^* a$ in $X'\setminus \pi^{-1} V$.  
Let $\chi_\epsilon=\chi(|h|^2v/\epsilon)$,
so that  $\pi^*\chi_\epsilon=\chi(|\pi^*h|\pi^*v/\epsilon)$. 
By Proposition~\ref{kraka} there is  $\mu'\in\PM(X')$ such that $\pi_*\mu'=\mu$. Thus
$$
\chi_\epsilon a\w\mu=\chi_\epsilon a\w\pi_*\mu'=
\pi_*\big(\pi^*\chi_\epsilon\pi^*a\w\mu'\big)=
\pi_*\big(\pi^*\chi_\epsilon \frac{\omega}{f}\w\mu'\big).
$$
In view of Proposition ~\ref{groda1} and Lemma~\ref{3apsko}, 
$$
\pi^*\chi_\epsilon \frac{\omega}{f}\w\mu'\to \1_{X'\setminus \pi^{-1}V}\frac{\omega}{f}\w\mu'
$$
when $\epsilon\to 0$.
In particular, the limit 
is a \pmm  current. Thus the limit in \eqref{asm4} exists and is pseudomeromorphic.  
\end{proof}

Notice that the definition of $a\wedge \mu$ is local, so that it commutes
with restrictions to open subsets of $X$. 
Thus for each $a\in ASM(X)$ we get a linear sheaf mapping
\begin{equation}\label{rav}
\PM_X \to \PM_X,  \quad
\mu\mapsto  a\wedge \mu.
\end{equation}

\begin{prop} Assume that $a\in ASM(X)$.
If $W$ is an analytic subset of $\U\subset X$ and  $\mu\in\PM(\U)$, then
\begin{equation}\label{pelargonia}
\1_W(a\w\mu)=a\w \1_W\mu.
\end{equation}
\end{prop}

\begin{proof}
On the one hand \eqref{pelargonia} holds in the open set $\U\setminus ZSS(a)$ by
\eqref{brutus1} since $a$ is
smooth there. On the other hand both sides vanish on $ZSS(a)$, so
\eqref{pelargonia} holds in all of $\U$; indeed $\1_{ZSS(a)} (a\w\1_W\mu)=0$ by definition, cf.,
Theorem ~\ref{hittills}, and 
$\1_{ZSS(a)}\1_W (a\w\mu)=\1_W \1_{ZSS(a)}(a\w\mu)=0$
in view of \eqref{skolgard}. 
%
%
\end{proof}

\begin{prop} \label{koko}
Each $a\in ASM(X)$ induces a linear mapping
\begin{equation}\label{rav2}
\W_X^Z \to \W_X^Z, \quad
\mu\mapsto  a\wedge \mu.
\end{equation}
\end{prop}

\begin{proof} To begin with, certainly $a\w\mu$ has support on $Z$ if $\mu$ has.  Let $\U$ be an open subset of $X$ and assume that
$W\subset \U\cap Z$ has positive codimension in $\U\cap Z$. Then
$\1_W(a\w\mu)=a\w \1_W\mu=0$ if $\1_W\mu=0$, cf., \eqref{pelargonia}. 
\end{proof}

\begin{ex} Assume that $\mu$ is in $\W_X$. Then  
$\mu':=[1/h]\mu$ is in $\W$ as well and if $h$ is generically nonvanishing, then 
$h\mu'=h[1/h]\mu=\1_{\{h\neq 0\}}\mu=\mu$, cf., Remark ~\ref{linne}.
\end{ex}

\begin{prop}
Assume that $a_1, a_2\in ASM (X)$ and $\mu\in\PM_X$. Then 
\begin{equation}\label{strosa}
a_1\w a_2\w \mu=(-1)^{\deg a_1 \deg a_2} a_2\w a_1 \w \mu. 
\end{equation}
\end{prop}

\begin{proof}
Notice that both sides of \eqref{strosa} coincide outside
$ZSS(a_1)\cup ZSS(a_2)$ and the restictions to  $ZSS(a_1)\cup
ZSS(a_2)$ vanish. 
\end{proof} 

In particular, one of the $a_j$ may be a smooth form.  We conclude that both
\eqref{rav} and \eqref{rav2} are $\E$-linear.

\begin{prop} If $a_1,a_2\in ASM(X)$ and $\mu\in \W_X$, then 
\begin{equation}\label{proms}
a_1\w a_2\w \mu=(a_1\w a_2)\w \mu, \quad (a_1+a_2)\w\mu=a_1\w\mu +
a_2\w\mu.  
\end{equation}
\end{prop}

In fact, \eqref{proms} holds outside $V:=ZSS(a_1)\cup ZSS(a_2)$ and
since $\1_V \mu=0$ the equalities follow from \eqref{pelargonia}.

\begin{ex}
Both equalities in \eqref{proms} may fail for a general $\mu\in\PM_X$. 
Let $a_1=1/z_1$, $a_2=z_1/z_2$, $a_3=1/z_2$, and 
$\mu=\dbar(1/z_1)$. Then $(a_1a_2)\mu=(1/z_2)\dbar(1/z_1)$, but
$a_2\mu=0$, and so $a_1a_2\mu=0$. 
Moreover 
\[
(a_1+a_3)\mu=\frac{z_2+z_1}{z_1z_2}\dbar\frac{1}{z_1}=0
\]
but 
\[
a_1\mu+a_3\mu=\frac{1}{z_1}\dbar\frac{1}{z_1}+\frac{1}{z_2}\dbar\frac{1}{z_1}=\frac{1}{z_2}\dbar\frac{1}{z_1}.
\]
\end{ex}

\subsection{Vector-valued almost semi-meromorphic currents}\label{lisa}
We will need to consider almost semi-meromorphic currents that take values in a holomorphic
vector bundle $E\to X$.  We say that $a\in ASM(X,E)$ if there is a representation \eqref{asm},
where as before $f$ is a holomorphic section
of $L\to X'$ and now $\omega$ takes values in $L\otimes \pi^* E$.  Clearly then $a$ is a current with values in $E$. If $\eta$ is a test form with values in
the dual bundle $E^*$, then $a.\eta=\pi_*( (\omega/f).\pi^*\eta )$. Let $e_j$ be a local frame for $E$
in $\U$ and let $\xi$ be a test function with support in $\U$. If $\xi'=\pi^*\xi$,
$e_j'=e_j\circ\pi$ and $\omega=\omega_1 e_1'+\omega_2 e'_2+\cdots$, then
\begin{equation}\label{fiol}
\xi a=\sum_j \pi_*(\xi'\omega_j/f) e_j.
\end{equation}

\begin{prop} \label{oden}
Assume that $X$ is smooth. There are natural isomorphisms 
\begin{equation}\label{flinga}
ASM^{p,*}(X,E)\simeq ASM^{0,*}(X,\Lambda^pT^*_{1,0}(X)\otimes E).
\end{equation}
\end{prop}

\begin{proof}
First notice that if $F,G$ are vector bundles of the same rank over $X'$ and $h$ is a holomorphic section of $\Hom(F,G)$
that is generically invertible, then there is a holomorphic section $g$ of
$\Hom(G,F)\otimes \det G \otimes (\det F)^{-1}$ such that 
$hg=s\cdot I_G$, where $s$  is a generically nonvanishing section of $\det G \otimes (\det F)^{-1}$.  

For simplicity we assume that $E$ is a trivial line bundle; the general case is proved in the same way.
Now, let $F=\pi^*\Lambda^pT^*_{1,0}(X)$ and $G=\Lambda^pT^*_{1,0}(X')$. Then
we have a natural mapping $h\colon F\to G$ as above, defined by
just mapping the
frame element $dz_I$ to its pullback $\pi^*dz_I$.  Clearly $h$ is an isomorphism where
$\pi\colon X'\to X$ is biholomorphic.  
 
Now, if $a\in ASM^{0,*}(X,\Lambda^pT^*_{1,0}(X))$, then we have the representation
$a=\pi_*(\omega/f)$,  where $\omega$ takes values in $F\otimes L$.  Then 
$h\omega$ is a $(p,*)$-form in $X'$ with values in $L$. It follows that 
$a':=\pi_*(h\omega/f)$ is an element in $ASM^{p,*}(X)$.  We claim that $a'=a$. By the
SEP it is enough to verify the identity where $\pi$ is a biholomorphism. 
Let $z$ be coordinates in an open subset $\U\subset X\setminus \pi(\sing\pi)$, and let
$\xi$ be a test function with support in $\U$. Then, cf., \eqref{fiol}, 
\begin{multline*}
\xi a=\sum'_{|I|=p}\pi_*(\xi'\omega_I/ f)\w dz_I=
\pi_*(\xi'\sum'_{|I|=p}\omega_I/ f\w \pi^*dz_I)=\pi_*(\xi' h\omega/f)=\\
\xi\pi_*(h\omega/f)
=\xi a'.
\end{multline*}

Conversely, since $h^{-1}=g/s$, if  $a'\in ASM^{p,*}(X)$, then
 $a'=\pi_*(\tilde\omega/f)$, where $\tilde\omega$ is
a $(p,*)$-form with values in $L$, then $g\tilde\omega$ takes values in $F\otimes
\det G \otimes (\det F)^{-1}\otimes L$ and $sf$ takes values in 
$\det G \otimes (\det F)^{-1}\otimes L$, so that $a=\pi_*(g\tilde\omega/sf)$ is an element
in $ASM^{0,*}(X,\Lambda^pT^*_{0,1}(X))$. Again one verifies that they coincide
in $X\setminus \pi(\sing\pi)$.
\end{proof}

Notice that if $p=1$, then $s$ is a section of the relative canonical bundle
$K_{X'/X}=K_{X'}\otimes \pi^* K_X^{-1}.$

\subsection{Residues of almost semi-meromorphic currents}\label{lisa2}
We shall now study the effect of $\partial$ and $\dbar$ on almost semi-meromorphic currents.

\begin{prop}\label{skrot} 
If $a\in ASM(X)$, then $\partial a\in ASM(X)$ and $b:=\1_{X\setminus
  ZSS (a)} \dbar a\in ASM(X)$. 
\end{prop}
 
Thus we have the decomposition 
\begin{equation}\label{trots}
\dbar a=b+r,
\end{equation}
where $r:=\1_{ZSS (a)} \dbar a$ has support on $ZSS(a)$.

\begin{proof}
Assume that $a=\pi_*(\omega/f)$ and let $D=D'+\dbar$ 
be a Chern connection on $L\to X'$. Then
$$
\partial a=\pi_*(\partial \frac{\omega}{f})=\pi_*\frac{f\cdot D'\omega
  - D'f\w\omega}{f^2}, 
$$
which is in $ASM(X)$. 

In view of Lemma ~\ref{asm0} we may assume that $Z(f)\subset \pi^{-1}V$, where
$V=ZSS(a)$. Now 
\begin{equation}\label{sprut}
\dbar a=\pi_*\frac{\dbar\omega}{f}+\pi_*\dbar\frac{1}{f}\w\omega.
\end{equation}
By \eqref{brutus2}, 
\begin{equation}
\1_{X\setminus V}~\dbar
a=\pi_*\left(\1_{\pi^{-1}(X\setminus V)}\frac{\dbar\omega}{f}\right)
+
\pi_*\left(\1_{\pi^{-1}(X\setminus V)} ~\dbar\frac{1}{f}\w\omega\right)
=
\pi_*\left(\frac{\dbar\omega}{f}\right); 
\end{equation}
thus $\1_{X\setminus V}~\dbar a \in ASM (X)$. 
For the last equality we have used Proposition ~\ref{groda1} and the
fact that $\dbar(1/f)$ has support on $\pi^{-1}V$. 
\end{proof}

In the same way we have: {\it If $a\in ASM(X,E)$  then \eqref{trots} holds,
where   $b=\1_{X\setminus ZSS (a)} \dbar a$ is in $ASM(X,E)$ and
$r=\1_{ZSS (a)} \dbar a$ is a pseudomeromorphic current with support on $ZSS(a)$ that takes
values in $E$.}

\smallskip

 Clearly the decomposition \eqref{trots} is unique. We call  $r=r(a)$ the
\emph{residue (current)} of $a$.  Notice that if 
$a$ is almost smooth, then $r(a)=0$.

\begin{remark} 
If $a=\pi_*(\omega/f)$ is any representation of $a$, then still \eqref{sprut} holds, and
since the first term is in $ASM(X)$ we conclude that
$$
r(a)=\pi_*\left ( \dbar \frac{1}{f} \w \omega \right ).
$$
\end{remark}

Notice that the current $\dbar (1/f)$ is the residue of the
principal value current $1/f$.  
Similarly, the residue currents introduced, e.g., in \cite{PTY, A1, AW1} can be
considered as residues of certain almost semi-meromorphic currents,
generalizing $1/f$.

\begin{ex}\label{asmex}
Let us describe the construction of the residue currents in \cite{A1}. 
Let $f$ be a holomorphic section of a Hermitian vector bundle $E\to X$, and let $\sigma$ be the section over $X\setminus Z(f)$ 
of the dual bundle $E^*$ with minimal norm such that $f\sigma=1$.  
We can find a 
modification $\pi\colon X'\to X$ that is a biholomorphism $X'\setminus\pi^{-1}Z(f)\simeq X\setminus Z(f)$
such that $\pi^*f=f^0 f'$, where $f^0$ is a holomorphic section of a line bundle $L\to X'$,
$\divi f^0$ is contained in  $\pi^{-1}Z(f)$, and
$f'$ is a nonvanishing section of $\pi^*E\otimes L^{-1}$. Then 
$$
\pi^*\sigma=\sigma'/f^0,
$$
where
$\sigma'$ is a smooth section of $\pi^*E^*\otimes L$.  
Thus 
$$
\pi^*\big(\sigma\w(\dbar\sigma)^{k-1}\big)=\frac{\sigma'\w(\dbar\sigma')^{k-1}}{(f^0)^k}
$$
is a section of $\Lambda^k(\pi^*E\oplus T^*_{0,1}(X'))$ in
$X'\setminus\pi^{-1}Z(f)$; for the reader's convenience note that
$\dbar\sigma$ has even degree in $\Lambda^k(\pi^*E\oplus T^*_{0,1}(X'))$.  It follows that
$$
U_k:=\sigma\w(\dbar\sigma)^{k-1}
$$
has an extension to an almost semi-meromorphic section of $\Lambda^k(E\oplus T^*_{0,1}(X))$,
as the push-forward of $\sigma'\w(\dbar\sigma')^{k-1}/(f^0)^{k}$.  Clearly
$ZSS(U_k)\subset Z(f)$. 
Now the residue current $R$ in \cite{A1} is the residue of the almost
semi-meromorphic current  $U=\sum_k U_k$.  More precisely, if $\delta_f$ denotes
interior multiplication by $f$, then 
$(\delta_f-\dbar) U=1-R$, i.e., $\dbar U=R+\delta_fU-1$, where $R$ is the residue and
$\delta_fU-1$ is almost semi-meromorphic.
If $E$ is trivial with trivial
metric, the coefficients of $R$ are the Bochner-Martinelli residue
currents introduced in \cite{PTY}. 
\end{ex}

\smallskip 
Clearly Theorem ~\ref{hittills} extends to vector-valued currents.
As a consequence of this theorem we can define products 
of residues of almost semi-meromorphic currents and pseudomeromorphic
currents:

\begin{df}\label{bounce}
For $a\in ASM (X,E)$ and $\mu\in \PM_X$ we define 
\begin{equation}\label{konc}
\dbar a\w\mu:=\dbar(a\w\mu)-(-1)^{\deg a} a\w\dbar\mu,
\end{equation} 
where $a\w\mu$ and $a\w\dbar\mu$ are defined as in
Theorem ~\ref{hittills}. 
Moreover we define
\begin{equation*}
r(a)\wedge \mu:= \1_{ZSS(a)}\dbar a\wedge \mu. 
\end{equation*}
\end{df}
Thus $\dbar a\w\mu$ is defined so that  the Leibniz rule holds. 
It is easily checked that
\begin{equation}\label{dagis} 
r(a)\wedge\mu=\lim_{\epsilon\to 0}\dbar \chi(|h|^2v/\epsilon)a\w\mu, 
\end{equation} 
if $Z(h)=ZSS(a)$. 
%
In particular this gives a way of defining products of $\dbar$ and residues of almost
semi-meromorphic currents. For example, the Coleff-Herrera product
$\dbar(1/f_1)\w\cdots\w\dbar(1/f_p)$ can be defined by inductively
applying \eqref{konc}. 
In \cite{Astrong} the first author defined 
products of more general residue currents in this way.

Notice that in general $a_1\w\dbar a_2$ is {\it not} equal to $\pm \dbar a_2\w a_1$, cf., Remark ~\ref{alligator}, and neither is 
\begin{equation}\label{pravda}
\dbar a_1\w\dbar a_2=\pm \dbar a_2\w \dbar a_1
\end{equation} 
in general; take, e.g., $a_1=1/z$ and $a_2=1/zw$. 

\begin{thm}\label{newthm}
Assume that $a_1,\ldots, a_p$ are almost semi-meromorphic currents of
degree $(*,k_1-1),\ldots, (*,k_p-1)$, respectively, and that
\begin{equation}\label{villkor}
\codim \Big (ZSS(a_{i_1})\cap\cdots\cap ZSS(a_{i_r}) \Big )\geq
k_{i_1}+\cdots + k_{i_r}
\end{equation}
for all    $\{i_1,\ldots, i_r\}\subset\{1,\ldots, p\}$. 
Then 
\begin{multline}\label{jojo}
\dbar a_1\w\cdots\w \dbar a_j\w \dbar a_{j+1}\w \cdots\w \dbar a_p=\\
(-1)^{(\deg a_j+1)(\deg a_{j+1}+1)}\dbar a_1\w\cdots\w \dbar a_{j+1}\w \dbar a_{j}\w \cdots\w \dbar a_p.
\end{multline}
\end{thm}
\begin{remark}\label{roso}
In fact, one can modify the proof below so that one can replace any
factor $\dbar
a_i$ in \eqref{jojo} by $a_i$. More precisely, 
let $b_i$ be either $a_i$ or $\dbar a_i$ for $i=1,\ldots, p$. Then 
\begin{equation}
b_1\w\cdots\w b_j\w b_{j+1}\w\cdots\w b_p=
(-1)^{\deg b_j \cdot \deg b_{j+1}} b_1\w\cdots\w b_{j+1}\w
b_{j}\w\cdots\w b_p. 
\end{equation}
\end{remark}

\begin{remark}
If the almost semimeromorphic parts of $\dbar a_i$ vanish, then it is
enough to assume 
\begin{equation}\label{secret}
\codim \Big (ZSS(a_{1})\cap\cdots\cap ZSS(a_{p}) \Big )\geq
k_{1}+\cdots + k_{p}. 
\end{equation} 
Indeed, note that in this case the currents in \eqref{jojo} have
support on $V:=ZSS(a_{1})\cap\cdots\cap ZSS(a_{p})$. Thus it is enough
to prove \eqref{jojo} in a neighborhood of $x\in V$, and there
\eqref{secret} implies \eqref{villkor}. 

In particular, the  Coleff-Herrera product 
$\dbar(1/f_1)\w\cdots\dbar(1/f_p)$
is (anti-)commutative in its factors if 
the codimension of $\{f_1=\ldots =f_p=0\}$ is at least $p$. 


\end{remark}

\begin{proof}
Let $V_j=ZSS(a_j)$. Moreover, let $b_i$ be either an almost
semi-meromorphic current or
$\dbar$ of an semi-meromorphic current for $i=1,\ldots, r$, cf., Remark \ref{roso}, and assume
that $\alpha$ is smooth. Then note that 
\begin{equation}\label{brev}
b_1\w\cdots\w b_\ell\w \alpha \w b_{\ell+1}\w\cdots\w b_r=
(-1)^{\deg\alpha (\deg b_1+\ldots +\deg b_\ell)} \alpha\w
b_1\w\cdots\w b_r. 
\end{equation} 
Assume that 
\begin{multline}\label{hoho}
\dbar a_1\w\cdots\w \dbar a_{j-1}\w a_j\w \dbar a_{j+1}\w \cdots\w \dbar a_p=\\
(-1)^{\deg a_j(\deg a_{j+1}+1)}\dbar a_1\w\cdots\w \dbar a_{j-1}
\w \dbar a_{j+1}\w a_{j}\w \dbar a_{j+2}\w \cdots\w \dbar a_p.
\end{multline}
Applying $\dbar$ to \eqref{hoho} yields \eqref{jojo} in
view of \eqref{konc}. 

To prove \eqref{hoho} we will proceed by induction. First assume that
$p=2$. Then in view of \eqref{brev}, 
\begin{equation}\label{berta}
a_1\w\dbar a_2=(-1)^{\deg a_1(\deg a_2+1)}\dbar a_2 \wedge a_1,
\end{equation}
where $a_1$ or $a_2$ is smooth, i.e., outside $V_1\cap V_2$. 
Because of the assumption
\eqref{villkor}, \eqref{berta} holds in all of $X$ by the dimension
principle. 
Next, assume that \eqref{hoho} holds for
$p=\ell$. In view of \eqref{brev}, \eqref{hoho} holds for $p=\ell+1$, where
$a_j$ or $a_{j+1}$ is smooth. 
Moreover, by \eqref{brev} and the assumption that 
\eqref{hoho} holds for $p=\ell$, \eqref{hoho} holds for $p=\ell+1$,
where (at least) one of $a_1,\ldots, a_{j-1}, a_{j+2}, \ldots,
a_{\ell+1}$ is smooth. Thus \eqref{hoho} holds for $p=\ell+1$ outside
$V_1\cap\cdots\cap V_{\ell+1}$, and thus by \eqref{villkor} and the dimension principle it
holds in all of $X$. Hence \eqref{hoho} and thus
\eqref{jojo} hold for all $p$. 
\end{proof}

The following example shows that $r(a)=0$ does not imply that
$r(a)\wedge \mu=0$. This points out the importance of keeping in mind
that $\mu\mapsto r(a)\w\mu$ is an operator on $\PM_X$ rather than a
"product".

\begin{ex}\label{cello}
Let us consider the setting in Example ~\ref{asmex}. Assume in addition
that $Z(f)$ has codimension at least $2$. Note that then $r(\sigma)=0$ by the dimension principle, since it has
bidegree $(0,1)$ and support on $Z(f)$, which has codimension $\geq
2$. 
However, if $\tau$ is the almost semi-meromorphic part of $\dbar U$,
then  
$r(\sigma)\wedge \tau$ is the residue current $R$ from
\cite{A1} which is nonzero,  cf.,  Example ~\ref{asmex}. 
\end{ex}

\begin{remark}
There are other (weighted) approaches to 
products of residue currents, see, e.g., \cite{P, W},
which 
coincide with the products above under suitable
conditions. 
\end{remark}

\subsection{Action of holomorphic differential operators and vector fields}
\label{actsec}

Finally we prove that $ASM(X)$ is preserved under the action of
holomorphic vector fields.

\begin{thm}\label{storknar} Let $\xi$ be a holomorphic vector field on a smooth manifold
$X$. 
If $a\in ASM(X)$, then the contraction $\xi\neg a$ and the Lie derivative
$L_\xi a$, a~priori defined on
$X\setminus  ZSS(a)$, have extensions as elements in $ASM(X)$.
\end{thm}

Since the extensions, if they exist, must be unique, we can simply say that
$\xi\neg a$ and $L_\xi a$ are in $ASM(X)$.

\begin{proof}
Let $\pi\colon X'\to X$ be a modification so that $a$ has the form
\eqref{asm}.  Then $\xi':=\pi^*\xi$ is a global section of $\pi^*T(X)$, that
is the natural lifting of $\xi$ to $T(X')$ over $X'\setminus\sing(\pi)$. 
By duality the mapping $\pi^*T^*_{1,0}(X)\to T^*_{1,0}(X')$  from the proof of Proposition ~\ref{oden}  induces a holomorphic mapping
$T(X')\to\pi^*T(X)$ that is the identity outside $\sing(\pi)$. If $h$
denotes this dual map, by the first part of the same proof there is a holomorphic mapping
$g\colon \pi^*T(X)\to T(X')\otimes K_{X'/X}$ such that $hg=sI_{\pi^*T(X)}$,
where $s$ is a holomorphic section of  $K_{X'/X}$.
Thus $g\xi'/s$ is a semi-meromorphic vector field on $X'$ that
coincides with $\xi'$ on $X'\setminus\sing(\pi)$. Moreover, $b:=s\xi'$ is smooth.
Outside
$\pi(\sing(\pi))\cup ZSS(a)$ we now have that
$$
\xi\neg a=\pi_*\big(\frac{\xi'\neg \omega}{f}\big)=\pi_*\big(\frac{b\neg \omega}{s f}\big)
$$
and it is clear that the right hand side defines an almost semi-meromorphic current
in $X$.
Finally, 
$L_\xi a=\xi\neg(\partial a)+\partial (\xi \neg a)$
is in $ASM(X)$ in view of Proposition~\ref{skrot}.
\end{proof}

By similar arguments one can prove that $\mathcal L a$ is in $ASM(X)$ if
$a$ is an almost semi-meromorphic $(0,q)$-current and $\mathcal L$ is any (global) holomorphic
differential operator. 
More precisely, one can show that $\mathcal L a=\pi_* (s^{-N} \mathcal L' (\omega/f))$ for some $N$, where $s$ is the section of
$K_{X'/X}$ in the proof above and $\mathcal L'$ is a
holomorphic differential operator (with values in $K_{X'/X}^N$).

\begin{cor} Let $X$ be an open subset of $\C^n_z$. 
If 
\begin{equation}\label{aform}
a=\sum_{|I|=p}' a_I\w dz_I
\end{equation}
is in $ASM(X)$, then  each  $a_I$ is in $ASM(X)$.
If $a\in ASM(X)$ has bidegree $(0,*)$, then 
$\partial a/\partial z_j$ is in  $ASM(X)$ for each  $j$.
\end{cor}

\end{document}